\documentclass[12pt]{article}
\usepackage{amsmath,amssymb,amsthm,amsfonts,amsbsy, enumerate,graphicx,graphics,epstopdf}
\usepackage[dvips]{epsfig}
\usepackage[T1]{fontenc}
\usepackage[utf8]{inputenc}
\usepackage{authblk}
\newcommand*{\affaddr}[1]{#1} 

\newcommand*{\email}[1]{\texttt{#1}}
\theoremstyle{plain}
\newtheorem{theorem}{Theorem}[section]

\newtheorem{example}[theorem]{Example}
\newtheorem{lemma}[theorem]{Lemma}

\newtheorem{rk}{Remark}

\newcommand{\rank}{{\rm rank}}

\newcommand{\F}{\mathbb{F}}
\allowdisplaybreaks

\begin{document}

\title{Spectral and Combinatorial Properties of Some Algebraically Defined Graphs}
\author{%
Sebastian M. Cioab\u{a},
 Felix Lazebnik, and Shuying Sun\\
\affaddr{\small{Department of Mathematical Sciences,  University of Delaware, \\Newark, DE 19707-2553, USA.}}
\vskip 2mm
\email{\small{cioaba@udel.edu,fellaz@udel.edu,shuying@udel.edu}}\\
}
\date{\today}
\maketitle

\begin{abstract} Let $k\ge 3$ be an integer, $q$ be a prime power, and $\mathbb{F}_q$ denote the field of $q$ elements.
Let $f_i, g_i\in\mathbb{F}_q[X]$, $3\le i\le k$, such that $g_i(-X) = -\, g_i(X)$.
We define a graph $S(k,q) = S(k,q;f_3,g_3,\cdots,f_k,g_k)$ as a graph with the vertex set $\F_q^k$ and edges  defined as follows: vertices $a = (a_1,a_2,\ldots,a_k)$ and $b = (b_1,b_2,\ldots,b_k)$ are adjacent if $a_1\ne b_1$ and the following $k-2$ relations on their components hold:
$$
b_i-a_i = g_i(b_1-a_1)f_i\Bigl(\frac{b_2-a_2}{b_1-a_1}\Bigr)\;,\quad 3\le i\le k.
$$
We show that graphs $S(k,q)$ generalize several recently studied examples of regular expanders and can provide many new such examples.
\end{abstract}

\section{Introduction and Motivation}\label{Intro}
All graphs in this paper are simple,  i.e., undirected, with no loops and no multiple edges. See, e.g.,  Bollob\'as \cite{Bol98} for standard terminology. Let $\Gamma = (V,E)$ be a graph with vertex set $V$ and edge set $E$. For a subset of vertices $A$ of $V$, $\partial A$ denotes the set of edges of $\Gamma$ with one endpoint in $A$ and the other endpoint in $V\setminus A$. The \emph{Cheeger constant} $h(\Gamma)$ (also known as \emph{edge-isoperimetric number} or \emph{expansion ratio}) of $\Gamma$, is defined by $h(\Gamma) := \min\Bigl\{\frac{|\partial{A}|}{	|A|} : A \subseteq V, 0 < |A| \le \frac{1}{2}|V| \Bigr\}.$
The graph $\Gamma$ is $d$-regular if each vertex is adjacent to exactly $d$ others. An infinite family of \emph{expanders} is an infinite family of regular graphs whose Cheeger constants are uniformly bounded away from 0. More precisely, for $n\geq 1$, let $\Gamma_n = (V_n, E_n)$ be a sequence of graphs such that each $\Gamma_n$ is $d_n$-regular and $|V_n| \rightarrow \infty$ as $n\rightarrow\infty$. We say that the members of the sequence form a \emph{family of expanders} if the corresponding sequence $\bigl(h(\Gamma_n)\bigr)$ is bounded away from zero, i.e. there exists a real number $c > 0$ such that $h(\Gamma_n) \ge c$ for all $n\geq 1$. In general, one would like the valency sequence $(d_n)_{n\geq 1}$ to be growing slowly with $n$, and ideally, to be bounded above by a constant. For examples of families of expanders, their theory and applications, see Davidoff, Sarnak and Valette \cite{DSV}, Hoory, Linial and Wigderson \cite{HLW},  and  Krebs and Shaheen \cite{KS11}.

The adjacency matrix $A=A(\Gamma)$ of a graph $\Gamma=(V,E)$ has its rows and columns labeled by $V$ and $A(x,y)$ equals the number of edges between $x$ and $y$ (i.e. 0 or 1).
When $\Gamma$ is simple, the matrix $A$ is symmetric and therefore, its eigenvalues are real numbers. For $j$ between $1$ and the order of $\Gamma$, let $\lambda_j=\lambda_j(G)$ denote the $j$-th eigenvalue of $A$. For an arbitrary graph $\Gamma$, it is hard to find or estimate $h(\Gamma)$, and often it is done by using  the second-largest eigenvalue $\lambda_2(\Gamma)$ of the adjacency matrix of $\Gamma$. If $\Gamma$ is a connected $d$-regular graph, then
$\frac{1}{2}\big(d-\lambda_2\big) \le h(\Gamma) \le \sqrt{d^2-\lambda_2^2}$. The lower bound was proved by Dodziuk \cite{Dod84} and independently by Alon-Milman \cite{ALM85} and by Alon \cite{AL86}. In both \cite{ALM85} and \cite{AL86},  the upper bound on $h(\Gamma)$, namely  $\sqrt{2d(d-\lambda_2)}$ was provided.  Mohar in \cite{Mo89} improved the upper bound to the one above. See \cite{BH11,HLW, KS11}, for terminology and results on spectral graph theory and connections between eigenvalues and expansion properties of graphs. The difference $d- \lambda_2$ which is present is both sides of this inequality above, also known as the \emph{spectral gap} of $\Gamma$, provides an estimate on the expansion ratio of the graph. In particular, for an infinite family of $d$-regular graphs $\Gamma_n$, the sequence $\bigl(h(\Gamma_n)\bigr)$ is bounded away from zero if and only if the sequence $\bigl(d-\lambda_2(\Gamma_n)\bigr)$ is bounded away from zero. A $d$-regular connected graph $\Gamma$ is called {\it Ramanujan} if $\lambda_2(\Gamma)\le 2\sqrt{d-1}$. Alon and Boppana \cite{Nil91} proved that  this bound is asymptotically best possible for any infinite family of $d$-regular graphs and their results imply that for any infinite family of $d$-regular connected graphs $\Gamma_n$, $\lambda_2(\Gamma_n) \ge 2\sqrt{d-1} - o_n(1).$ For functions $f, g : \mathbb{N}\rightarrow \mathbb{R}^+$, we write $f = o_n(g)$ if $f(n)/g(n) \rightarrow 0$ as $n\rightarrow \infty$.

\bigskip
For the rest of the paper, let $q=p^e$, where $p$ is a prime and $e$ is a positive integer. For a sequence of prime powers $(q_m)_{m\ge 1}$, we always assume that $q_m = p_m^{e_m}$, where $p_m$ is a prime and $e_m\ge 1$. Let $\F_q$ be the finite field of $q$ elements and $\F_q^k$ be the cartesian product of $k$ copies of $\F_q$. Clearly, $\F_q^k$ is a vector space of dimension $k$ over $\F_q$.
For $2\leq i\leq k$, let $h_i$ be an arbitrary polynomial in $2i-2$ indeterminants over $\F_q$.
We define the bipartite graph
$B\Gamma_k =
B\Gamma (q;h_2,\ldots, h_k)$, $k\ge 2$,
as follows.
The vertex set of  $B\Gamma_k$ is the
disjoint union of two copies of
$\F_q^k$, one denoted by $P_k$ and the other by $L_k$.
We define edges of $B\Gamma_k$ by declaring vertices
$p=(p_1,p_2,\ldots,p_k)\in P_k$ and
$l=(l_1,l_2,\ldots,l_k)\in L_k$  to be adjacent
if  the following
$k-1$ relations on their coordinates hold:
\begin{equation}\label{Bipmain}
p_i + l_i = h_i(p_1,l_1,p_2,l_2, \ldots, p_{i-1},l_{i-1}), \;i=2,\ldots, k.
\end{equation}
The graphs $B\Gamma_k$ were introduced by Lazebnik and Woldar \cite{LW01}, as generalizations of graphs introduced by  Lazebnik and Ustimenko in \cite{LazUst} and \cite{LazUstDkq}. For surveys on these graphs and their applications, see
\cite{LW01} and Lazebnik, Sun and  Wang \cite{LaSuWa17}. An important  basic
property of graphs $B\Gamma_k$ (see  \cite{LW01}) is that
for every vertex $v$ of $B\Gamma_k$ and every
$\alpha \in \F_q$,  there exists a unique
neighbor of $v$  whose first coordinate is $\alpha$. This implies that each $B\Gamma_k$ is $q$-regular, has $2q^k$ vertices and $q^{k+1}$ edges.

The spectral and combinatorial properties of three specializations of graphs $B\Gamma_k$ has received particular attention in recent years. Cioab\u{a}, Lazebnik and Li \cite{CiLaLi14} determined the complete spectrum of the {\it Wenger} graphs $W_{k}(q) = B\Gamma (q;h_2,\ldots, h_{k+1})$  with $h_i= p_1l_1^{i-1}$, $2\le i \le k+1$. Cao, Lu, Wan, Wang and  Wang  \cite{Liping15} determined the eigenvalues of the {\it linearized Wenger} graphs $L_{k}(q) = B\Gamma (q;h_2,\ldots, h_{k+1})$ with $h_i= p_1^{p^{i-2}}l_1$, $2\le i \le k+1$, and Yan and Liu \cite{YanLiu17} determined the multiplicities of the eigenvalues the linearized Wenger graphs. Moorhouse, Sun and Williford \cite{MoorSunWil16} studied the spectra  of  graphs $D(4,q) = B\Gamma (q; p_1l_1, p_1l_2, p_2l_1)$, and in particular, proved that  the second largest eigenvalues of these graphs are bounded from above by $2\sqrt{q}$ (so $D(4,q)$ is `close' to being Ramanujan).
\bigskip

Let $V_1$ and $V_2$  denote the partite sets or color classes of the vertex set of a bipartite graph $\Gamma$. The {\em distance-two graph of $\Gamma$ on $V_1$} is the graph having $V_1$ as its vertex set with the adjacency defined as follows: two vertices $x\neq y\in V_1$ are adjacent if there exists a vertex $z\in V_2$ adjacent to both $x$ and to $y$ in $\Gamma$ (which is equivalent of saying that $x$ and $y$ are at distance two in $\Gamma$). If $\Gamma$ is $d$-regular and contains no 4-cycles, then $\Gamma^{(2)}$ is a $d(d-1)$-regular simple graph. There is simple connection between the eigenvalues of $\Gamma$ and the eigenvalues of $\Gamma^{(2)}$ (see, e.g., \cite{CiLaLi14}):  every eigenvalue $\lambda$ of $\Gamma^{(2)}$  with multiplicity $m$, corresponds to a pair of eigenvalues $\pm\sqrt{\lambda + d}$ of $\Gamma$, each with multiplicity $m$ (or a single eigenvalue $0$ of multiplicity $2m$ in case $\lambda = -d)$.

This relation between the spectra of $q$-regular bipartite graph $\Gamma$ and its $q(q-1)$-regular distance-two graph $\Gamma^{(2)}$ has been utilized in each of the papers \cite{CiLaLi14,Liping15,MoorSunWil16} in order to find or to bound  the second-largest eigenvalue of $\Gamma$, and then use this information  to claim the  expansion property of $\Gamma$. In each of these cases, $\Gamma^{(2)}$  turned out to be a Cayley graph of a group, that allowed to use representation theory to compute its spectrum. In  \cite{CiLaLi14,Liping15} the group turned out be abelian, as in \cite{MoorSunWil16} it was not for odd $q$.

The main motivation behind the construction below is to directly generalize the defining systems  of equations for $W_k^{(2)}(q)$ and of $L_k^{(2)}(q)$, thereby obtaining a family of $q(q-1)$-regular Cayley graphs of an abelian group. The adverb {\em directly} used in the previous sentence was to stress that the graphs we build are not necessarily distance-two graphs of $q$-regular bipartite graphs~$\Gamma$. Examples when they are not will be discussed in Remark \ref{NonBip} of Section \ref{CONREM}.
\bigskip

\section{Main Results}\label{MR}

In this section, we define the main object of this paper, the family of graphs $S(k,q)$ and we describe our main results. Let $k$ be an integer, $k\ge 3$.
Let $f_i, g_i\in\mathbb{F}_q[X]$, $3\le i\le k$,  be $2(k-2)$ polynomials of degrees at most $q-1$ such that $g_i(-X) = -\,g_i(X)$ for each $i$.
We define $S(k,q) = S(k,q; f_3,g_3,\cdots,f_k,g_k)$ as the graph with the vertex set $\F_q^k$ and edges  defined as follows: $a = (a_1,a_2,\ldots,a_k)$ is adjacent to $b = (b_1,b_2,\ldots,b_k)$ if $a_1\ne b_1$ and the following $k-2$ relations on their coordinates hold:
\begin{equation}\label{S2SS1:relation}
b_i-a_i = g_i(b_1-a_1)f_i\Bigl(\frac{b_2-a_2}{b_1-a_1}\Bigr)\;,\quad 3\le i\le k.
\end{equation}

Clearly,  the requirement $g_i(-X)= -\, g_i(X)$ is used for the definition of the adjacency in $S(k,q)$ to be symmetric.
One can easily see that $S(k,q)$ is a Cayley graph
with the {\it underlying group} $G$ being  the additive group of the vector space $\F_q^k$ with {\it generating set}
$$
\left\{ \bigl(a, au, g_3(a)f_3(u), \cdots, g_k(a)f_k(u)\bigr) \ |\ a\in\F_q^*, u\in\F_q\right\}.
$$
This implies that $S(k,q)$ is vertex transitive of degree $q(q-1)$.

Note that for $f_i= X^{i-1}$ and $g_i= X$, $3\le i\le k+1$, $S(k+1,q)= W^{(2)}_k(q)$ is the distance-two graph of the Wenger graphs $W_k(q)$ on lines and for  $f_i=X^{p^{i-2}}$ and $g_i= X$, $3\le i\le k+1$, $S(k+1,q)= L^{(2)}_k(q)$ is the distance-two graph of the linearized Wenger graphs $L_k(q)$ on lines.
\bigskip

In order to present our results, we need a few more notation.  For any $\alpha\in\F_q$, let $Tr(\alpha) = \alpha+ \alpha^p + \cdots +\alpha^{p^{e-1}}$ be the {\em trace of $\alpha$ over $\F_p$}. It is known that $Tr(\alpha)\in \F_p$. For any element $\beta\in \F_p$, let $\beta^*$ denote the unique integer such that $0\le \beta^*<p$ and the residue class of $\beta^*$ in $\F_p$ is $\beta$.   For any complex number $c$, the expression $c^\beta$ will mean  $c^{\beta^*}$.
Let $\zeta_p =\exp(\frac{2\pi}{p}i)$ be a complex $p$-th root of unity.
For every $f\in\F_q[X]$, we call $\varepsilon_{f} = \sum\limits_{x\in\F_q} \zeta_p^{\ Tr\bigl(f(x)\bigr)}$ the {\em exponential sum of $f$}.

We are ready to state the main results of this paper.

The following theorem describes the spectrum of the graphs $S(k,q)$.
\begin{theorem}\label{a}
Let $k\ge 3$.
Then the spectrum of $S(k,q)$ is the multiset $\{ \lambda_w\  |\  w = (w_1,\cdots, w_k) \in\mathbb{F}_q^k\}$, where
\begin{equation}\label{S2SS4:eignoclosed}
\lambda_w = \sum\limits_{a\in\mathbb{F}_q^*,u\in\mathbb{F}_q} \zeta_p^{Tr\Big(aw_1+auw_2+\sum\limits_{i=3}^k g_i(a)f_k(u)w_k\Big)}.
\end{equation}
\end{theorem}
For a fixed $k\geq 3$, the theorem below provides sufficient conditions for the graphs $S(k,q)$ to form a family of expanders.
\begin{theorem} \label{aa}
Let $k\ge 3$, $(q_m)_{m\geq 1}$ be an increasing
 sequence of prime powers, 
and let
$$S(k,q_m) = S(k,q_m; f_{3,m}, g_{3,m},\cdots,f_{k,m},g_{k,m}).$$
Set $d_f^{(m)}=\max\limits_{3\le i\le k} {\rm deg}(f_{i,m})$ and $d_g^{(m)} = \max\limits_{3\le i\le k} {\rm deg}(g_{i,m})$.
Suppose $1\le d_f^{(m)} = o_m(q_m)$, $d_g^{(m)} = o_m(\sqrt{q_m})$, $1\le d_g^{(m)}< p_m$, and for all $m\ge 1$, at least one of the following two conditions   
is satisfied:
\begin{enumerate}
\item The polynomials $1,X,f_{3,m},\ldots,f_{k,m}$ are $\F_q$-linearly independent, and $g_{i,m}$ has linear term for all $i$, $3\le i\le k$.
\item The polynomials $f_{3,m},\ldots,f_{k,m}$ are $\F_q$-linearly independent, and there exists some $j$, $2\le j\le d_g^{(m)}$, such that each polynomial $g_{i,m}$, $3\le i\le k$, contains a term $c_{i,j}^{(m)}X^j$ with $c_{i,j}\ne 0$.
\end{enumerate}
Then $S(k,q_m)$ is connected and $\lambda_2\bigl(S(k,q_m)\bigr) = o_m(q_m^2)$.
\end{theorem}

The following two theorems demonstrate that for some specializations of $S(k,q)$, we can obtain stronger upper bounds on their second largest eigenvalues.

\begin{theorem}\label{aaaa}
Let $q$ be an odd prime power with $q \equiv 2\mod 3$, and $4\le k\le q+1$. Let $g_i(X) = X^3$ and $f_i(X) = X^{i-1}$ for each $i$, $3\le i\le k$. Then $S(k,q)$ is connected, and
$$
\lambda_2\bigl(S(k,q)\bigr) =\max\bigl\{q(k-3),(q-1)M_q\bigr\},
$$
where $M_q = \max\limits_{a,b\in\F_q^*}\varepsilon_{ax^3+bx}\le 2\sqrt{q}$.
\end{theorem}
For large $k$, specifically, when $(q-1)M_q \le q(k-3)$,
$$ \lambda_2\bigl(S(k,q)\bigr) = q(k-3) < q(k-2) = \lambda_2\bigl(W^{(2)}_{k-1}(q)\bigr).$$

Similarly to Theorem \ref{aaaa}, when choosing $f_i(X) = X^{p^{i-2}}$, the same $f$ functions as in $L^{(2)}_k(q)$, we obtain the following upper bounds for the second largest eigenvalue.
\begin{theorem}\label{aaaaa}
Let $q$ be an odd prime power with $q \equiv 2\mod 3$, and $3\le k\le e+2$. Let $g_i(X) = X^3$ and $f_i(X) = X^{p^{i-2}}$ for each $i$, $3\le i\le k$. Then $S(k,q)$ is connected, and
$$
\lambda_2\bigl(S(k,q)\bigr) \le \max\bigl\{q(p^{k-3}-1),(q-1)M_q\bigr\},
$$
where $M_q = \max\limits_{a,b\in\F_q^*}\varepsilon_{ax^3+bx}\le 2\sqrt{q}$.
\end{theorem}
For large $k$, specifically, when $(q-1)M_q \le q(p^{k-3}-1)$,
$$
\lambda_2\bigl(S(k,q)\bigr) = q(p^{k-3}-1) < q(p^{k-2}-1) = \lambda_2\bigl(L^{(2)}_{k-1}(q)\bigr).
$$

The paper is organized as follows.  In Section \ref{FF}, we present necessary definitions and results concerning finite fields used in the proofs. In Section \ref{EIGEN}, we prove Theorem~\ref{a}. In Section \ref{CONEXP}, we study some sufficient conditions on $f_i$ and $g_i$ for the graph $S(k,q)$ to be connected and have large eigenvalue gap, and prove Theorem \ref{aa}. In Section \ref{ODDEXP}, we prove Theorem \ref{aaaa} and Theorem~\ref{aaaaa}. We conclude the paper with several remarks in Section \ref{CONREM}.

\section{Background on finite fields}\label{FF}

For  definitions and theory of finite fields, see Lidl and Niederreiter \cite{LN97}.
\begin{lemma}[\cite{LN97}, Ch.5]\label{exp_sums}
If $f(X)=bX+c\in \F_q[X]$ is a polynomial of degree one or less, then
\[\varepsilon_f=\left\{\begin{array}{ll}
0,&\hbox{if $b\neq0$,}\\
q\zeta^{Tr(c)},&\hbox{otherwise.}\end{array}\right.\]
\end{lemma}

For a general $f\in \F_q[X]$, no explicit expression for the exponential sum $\varepsilon_f$ exists. The following theorem provides a good upper bound for the exponential sum $\varepsilon_f$.
\begin{theorem}[Hasse-Davenport-Weil Bound, \cite{LN97}, Ch.5]\label{WeilBound}
Let $f\in\mathbb{F}_q[X]$ be a polynomial of degree $n\ge 1$. If $gcd(n,q) = 1$, then
  $$
 | \varepsilon_f | \le (n-1) q^{1/2}.
  $$
  \end{theorem}

\begin{lemma}\label{oddexp}
  Suppose that $g\in\F_q[X]$ and $g(-X) =-\,g(X)$. Then $\varepsilon_g$ is a real number.
 \end{lemma}
\begin{proof}We have that
\begin{align*}
\varepsilon_g &= \sum\limits_{a\in\F_q}\zeta_p^{\ Tr(g(a))}= 1+\sum\limits_{a\in\F_q^*}\zeta_p^{\ Tr(g(a))}= 1+\frac{1}{2}\sum\limits_{a\in\F_q^*}\left (\zeta_p^{\ Tr(g(a))}+\zeta_p^{\ Tr(g(-a))}\right)\\&= 1+\frac{1}{2}\sum\limits_{a\in\F_q^*}\left (\zeta_p^{\ Tr(g(a))}+\zeta_p^{\ Tr(-g(a))}\right)= 1+\frac{1}{2}\sum\limits_{a\in\F_q^*}\left (\zeta_p^{\ Tr(g(a))}+\zeta_p^{\ -Tr(g(a))}\right).
\end{align*}
Since $\zeta_p^{\beta}+\zeta_p^{-\beta}\in\mathbb{R}$ for any $\beta\in\F_p$, it follows that $\varepsilon_g\in \mathbb{R}$.
\end{proof}

\section{Spectra of the graphs $S(k,q)$}\label{EIGEN}
The proof we present here is based on the same idea as the one in \cite{CiLaLi14}. Namely, computing eigenvalues of Cayley graphs by using the method suggested in Babai~\cite{Babai}.
 The original completely different (and much longer) proof of Theorem \ref{a} that used circulants appears in Sun \cite{thesis}.
\begin{theorem}[\cite{Babai}]\label{S3SS2:CayleySpectrum}
Let $G$ be a finite group and $S\subseteq G$ such that $1\not\in S$ and $S^{-1} = S$. Let $\{\pi_1,\ldots,\pi_k\}$ be a representative set of irreducible $\mathbb{C}$-representations of $G$.
 Suppose that the multiset $\Lambda_i := \{\lambda_{i,1},\lambda_{i,2},\ldots,\lambda_{i,n_i}\}$ is the spectrum of the complex $n_i\times n_i$ matrix $\pi_i(S) = \sum\limits_{s\in S}\pi_i(s) $. Then the spectrum of the  Cayley graph $X =$ Cay$(G,S)$ is the multiset formed as the union of $n_i$ copies of $\Lambda_i$ for $i\in \{1,2,\ldots,k\}$.
\end{theorem}
\begin{proof}[\textbf{Proof of Theorem \ref{a}}]
As we mentioned in Section \ref{MR}, $S(k,q)$ is a Cayley graph with the underlying group $G$ being  the additive group of the vector space $\F_q^k$, and connection set
$$
\left\{ \bigl(a, au, g_3(a)f_3(u), \cdots, g_k(a)f_k(u)\bigr) \ |\ a\in\F_q^*, u\in\F_q\right\}.
$$
Since $G$ is an abelian group, it follows that the irreducible $\mathbb{C}$-representations of $G$ are linear (see \cite{Isaac}, Ch. 2). They are given by
$$
\pi_w(v) = [\zeta_p^{\ Tr(w_1v_1+\cdots+w_kv_k) }],
$$
where $w = (w_1,\cdots,w_k)\in\F_q^k$ and $v = (v_1,\cdots,v_k)\in\F_q^k$.

Using  Theorem \ref{S3SS2:CayleySpectrum}, we conclude that the spectrum of $S(k,q)$ is a multiset formed by all $\lambda_w$, $w = (w_1,\cdots,w_k)\in\F_q^k$, of the form:
\begin{align*}
\lambda_w &= \sum\limits_{s\in S} \zeta_p^{\ Tr(w_1s_1+\cdots+w_ks_k)}\\
                &=\sum\limits_{a\in\mathbb{F}_q^*,u\in\mathbb{F}_q} \zeta_p^{\ Tr\bigl(aw_1+auw_2+\sum\limits_{i=3}^k g_i(a)f_i(u)w_i\bigr)}.
\end{align*}
\end{proof}

\section{Connectivity and expansion of the graphs $S(k,q)$}\label{CONEXP}

It is hard to get a
 closed form of $\lambda_w$ in (\ref{S2SS4:eignoclosed}) for arbitrary $f_i$ and $g_i$.
But if the degrees of the polynomials $f_i$ and $g_i$ satisfy some  conditions, we are able to show that the components of the graphs $S(k,q)$ have large eigenvalue gap. For these $f_i$ and $g_i$, we find sufficient conditions such that the graphs $S(k,q)$ are connected, and hence form a family of expanders.

From now on,
 for any graph $S(k,q; f_3,g_3,\cdots,f_k,g_k)$, we let $d_{g} = \max\limits_{3\le i\le k} \deg(g_i)$ and $d_{f} = \max\limits_{3\le i\le k} \deg(f_i)$. We also assume that $d_f \ge 1$ and $d_g\ge 1$.
 For each $i$, $3\le i\le k$, let $c_{i,j}$ be the coefficient of $X^j$ in
the polynomial $g_i$, for any $j$, $1\le j\le d_{g}$, i.e.
$$g_{i}(X) = c_{i,1}X+c_{i,2}X^2+\ldots+c_{i,d_{g}}X^{d_{g}}.$$

For any $w = (w_1,\cdots, w_k)$ in $\F_q^k$, let $N_w$ be the number of $u$'s in $\mathbb{F}_q$ satisfying the following system
  \begin{align}\label{S2SS4:N1}
  w_1+uw_2 +\sum\limits_{i=3}^{k} c_{i,1}f_{i}(u)w_{i}&= 0,\nonumber\\
  \sum\limits_{i=3}^{k} c_{i,j}f_{i}(u)w_{i} &= 0 , \hskip 10mm {2 \le j \le d_{g}}
\end{align}
and let
$S_w$
 be the set of all $u$'s in $\F_q$ such that the following inequality holds for some $j$, $2\le j\le d_{g}$,
\begin{equation}\label{S2SS4:N2}
  \sum\limits_{i=3}^{k} c_{i,j}f_{i}(u)w_{i} \ne 0.
\end{equation}
If $d_g = 1$, then system (\ref{S2SS4:N1}) contains only the first equation, and $S_w =\emptyset$.

\begin{lemma}\label{lem}
Let $k\ge 3$. If $1\le d_{g} < p$,
 then for any $w = (w_1,\cdots,w_k)$ in $\F_q^k$, the eigenvalue $\lambda_w$ of $S(k,q)$ in (\ref{S2SS4:eignoclosed}) is at most
\begin{equation}\label{S2SS4:b}
N_{w}(q-1)+ |S_{w}|[(d_{g}-1)\sqrt{q}+1].
\end{equation}
Moreover, $\lambda_w = q(q-1)$ if and only if $N_w  = q$.
\end{lemma}

\begin{proof}
Let $w = (w_1,\ldots,w_k)\in\mathbb{F}_q^k$. Using Theorem \ref{a}, we have
\begin{align*}
\lambda_{w}
     &= \sum\limits_{u\in\mathbb{F}_q}\sum\limits_{a\in\mathbb{F}_q^*} \zeta_p^{\ Tr\Big(a(w_{1}+uw_{2})+\sum\limits_{i=3}^{k}g_{i}(a)f_{i}(u)w_{i}\Big)}\\
     &= \sum\limits_{u\in\F_q}\sum\limits_{a\in\F_q^*} \zeta_p^{\ Tr\Big(
     a\bigl[w_1+uw_2+\sum\limits_{i=3}^{k}c_{i,1}f_{i}(u)w_{i}\bigr]
+a^2\sum\limits_{i=3}^{k}c_{i,2}f_{i}(u)w_{i}+\cdots+a^{d_{g}}\sum\limits_{i=3}^{k}c_{i,d_{g}}f_{i}(u)w_{i}
     \Big)}\\
     &= \sum\limits_{u\in\mathbb{F}_q} z_u,
  \end{align*}
where
$$z_u= \sum\limits_{a\in\F_q^*} \zeta_p^{\ Tr\Big(
     a\bigl[w_1+uw_2+\sum\limits_{i=3}^{k}c_{i,1}f_{i}(u)w_{i}\bigr]
+a^2\sum\limits_{i=3}^{k}c_{i,2}f_{i}(u)w_{i}+\cdots+a^{d_g}\sum\limits_{i=3}^{k}c_{i,d_g}f_{i}(u)w_{i}
     \Big)}.
  $$
If $u$ satisfies (\ref{S2SS4:N1}), then $z_u = q - 1$.
If $u\in S_w$, then $z_u$ is an exponential sum of a polynomial of degree at least 2 and at most $d_g$. By the assumption of the theorem that $d_g < p$ and Weil's bound in Theorem \ref{WeilBound}, it follows that
$$
|z_u| \le (d_g-1)\sqrt{q}+1.
$$
Finally, for the remaining $q - N_w - |S_w|$ elements $u\in\F_q$,
we have
\begin{align}\label{S2SS4:N3}
  w_1+uw_2 +\sum\limits_{i=3}^{k} c_{i,1}f_{i}(u)w_{i}&\ne 0\nonumber\\
  \sum\limits_{i=3}^{k} c_{i,j}f_{i}(u)w_{i} &= 0 , \hskip 10mm 2 \le j \le d_g.
\end{align}
If $d_g = 1$, then system (\ref{S2SS4:N3}) contains only the first inequality. In both cases, we have $z_u = -1$.
Therefore, we have
\begin{align*}
\lambda_{w} &= N_w(q-1)+\sum\limits_{u\in S_w} z_{u} + (q-N_w-|S_w|)(-1)\\
   &\le (N_w-1)q+|S_w|[(d_g-1)\sqrt{q}+2]\\
                    &\le N_w(q-1) + |S_w|[(d_g-1)\sqrt{q}+1].
 \end{align*}
Let us now prove the second statement of the lemma.
It is clear that if $N_w = q$, then $|S_w|=0$
and $\lambda_w = q(q-1)$.
For the rest of this proof,  we assume that $N_w < q$, and show that  $\lambda_w < q(q-1)$.

If $e > 1$, then $(d_g-1)\sqrt{q}+1 < q-1$ as $d_g < p$. Therefore, $\lambda_w < q(q-1)$.

 For $e = 1$, we consider the following two cases: $q = p = 2$ and $q = p\ge 3$.

 If  $q = p = 2$, then $d_g = 1$ as $d_g < p$, and hence $|S_w| = 0$. Therefore, $\lambda_w < q(q-1)$.

If $q= p \ge 3$, then, as $\lambda_w$ is a real number and $|z_u| \le p-1$, we have
$$\lambda _w \le |\lambda_w| = |\sum_{u\in\F_p} z_u|  \le \sum_{u\in\F_p} |z_u| \le p(p-1),$$
and $\lambda _w = p(p-1)$ if and only if  $z_u=p-1$ for all  $u\in\F_p$.  The latter condition is equivalent to
$$Tr\Big(
     a\bigl[w_1+uw_2+\sum\limits_{i=3}^{k}c_{i,1}f_{i}(u)w_{i}\bigr]
+a^2\sum\limits_{i=3}^{k}c_{i,2}f_{i}(u)w_{i}+
\cdots+a^{d_g}\sum\limits_{i=3}^{k}c_{i,d_g}f_{i}(u)w_{i}
     \Big)=0 $$
for all $u\in\F_p$.   For $x\in \F_p$, $Tr(x) = 0$ if and only if $x=0$. This implies  that
$$a\bigl[w_1+uw_2+\sum\limits_{i=3}^{k}c_{i,1}f_{i}(u)w_{i}\bigr]
+a^2\sum\limits_{i=3}^{k}c_{i,2}f_{i}(u)w_{i}+\cdots+a^{d_g}\sum\limits_{i=3}^{k}c_{i,d_g}f_{i}(u)w_{i}     =0$$
for any $a\in \F_p^*$.
Therefore, the polynomial
$$X\bigl[w_1+uw_2+\sum\limits_{i=3}^{k}c_{i,1}f_{i}(u)w_{i}\bigr]
+X^2\sum\limits_{i=3}^{k}c_{i,2}f_{i}(u)w_{i}+\cdots+X^{d_g}\sum\limits_{i=3}^{k}c_{i,d_g}f_{i}(u)w_{i},$$
which is over $\F_p$, has $p$ distinct roots in $\F_p$ and is of degree at most $d_g$, $d_g< p$.  Hence,  it must be zero polynomial,  and so $N_p=p$, a contradiction. Hence, $\lambda_w <p(p-1)$.
\end{proof}

Let $(q_m)_{m\geq 1}$ be an increasing
 sequence of prime powers. For a fixed $k$, $k\ge 3$, we consider
 an infinite family of graphs $S(k,q_m; f_{3,m},g_{3,m},\cdots, f_{k,m}, g_{k,m})$. Hence,  $|V\bigl(S(k,q_m)\bigr)|  = q_m ^k\rightarrow \infty$ when $m\to \infty$. 
Let
$d_f^{(m)} = \max\limits_{3\le i\le k} \deg(f_{i,m})$ and $d_g^{(m)} = \max\limits_{3\le i\le k} \deg(g_{i,m})$, for each $m$.
In what follows we present conditions on $d_f^{(m)}$ and $d_g^{(m)}$  which imply
that the components of these graphs have large eigenvalue gaps.
\begin{theorem}\label{S2SS4:largegapgeneral}
Let $(q_m)_{m\geq 1}$ be an increasing
 sequence of prime powers.  Suppose that $d_f^{(m)} \ge 1$ and $1\le d_g^{(m)} < p_m$ for all $m$.
Let $\lambda^{(m)}$ be the largest eigenvalue of $S(k,{q_m})$ which is not $q_m(q_m-1)$ for any $m$.
Then
$$
\lambda^{(m)} = \max\big(O(d_f^{(m)}q_m),O(d_g^{(m)}q_m^{3/2})\big).
$$
\end{theorem}
\begin{proof}
For any $w \in\F_{q_m}^k$, the eigenvalue $\lambda_{w}$ of $S(k,q_m)$ is at most
$$N_{w}(q_m-1) + |S_{w}|[(d_g^{(m)}-1)\sqrt{q_m}+1],$$
by Lemma \ref{lem}.

It is clear that for any $w\in \F_{q_m}^k$,
  system (\ref{S2SS4:N1}) has either $N_w = q_m$ solutions or at most $d_f^{(m)}$ solutions with respect to $u$. If $N_w = q_m$, then $\lambda_w = q_m(q_m-1)$ by Lemma \ref{lem}. If
  $N_{w} < q_m$,
   then $N_{w}\le d_f^{(m)}$. Therefore, we have
\begin{align*}
\lambda_{w} &\le d_f^{(m)} (q_m-1) + q_m[(d_g^{(m)}-1)\sqrt{q_m}+1] \\
                        &\le d_f^{(m)}q_m+d_g^{(m)}q_m^{3/2} = \max\big(O(d_f^{(m)}q_m),O(d_g^{(m)}q_m^{3/2})\big).
\end{align*}
\end{proof}
As an immediate  corollary from Theorem \ref{S2SS4:largegapgeneral}, we have the following theorem.
\begin{theorem}\label{S2SS4:largegap}
Let $(q_m)_{m\geq 1}$ be an increasing
 sequence of prime powers.
Suppose that $1\le d_f^{(m)}= o_m(q_m)$, $d_g^{(m)}= o_m(\sqrt{q_m})$ and $1\le d_g^{(m)}< p_m$ for all $m$.
Let $\lambda^{(m)}$ be the largest eigenvalue of $S(k,{q_m})$ which is not $q_m(q_m-1)$ for any $m$.
Then
$$
\lambda^{(m)} = o_m(q_m^2).$$
\end{theorem}

Our
 next theorem provides a sufficient condition for the graph $S(k,q)$ to be connected.
\begin{theorem}\label{S2SS4:connectedness}
For $k\ge 3$,   let $S(k,q) = S(k,q;f_3,g_3,\cdots,f_k,g_k)$ and $1
\le d_g < p.$
If at least one of the following two conditions is satisfied, then $S(k,q)$ is connected.
\begin{enumerate}
\item The polynomials $1,X,f_3,\ldots,f_k$ are $\F_q$-linearly independent, and $g_i$ contains a linear term for each $i$, $3\le i\le k$.\\
\item The polynomials $f_3,\ldots,f_k$ are $\F_q$-linearly independent,
and there exists some $j$, $2\le j\le d_g$, such that each polynomial $g_i$, $3\le i\le k$, contains a term $c_{i,j}X^j$ with $c_{i,j}\ne 0$.
\end{enumerate}
\end{theorem}

\begin{proof}
First, notice that the number of components of $S(k,q)$ is equal to the multiplicity of the eigenvalue $q(q-1)$. By Lemma \ref{lem},  this  multiplicity is equal to
$\left|\{ w\in \F_q^k:  N_w = q\}\right|$. As the equality $N_w=q$ is equivalent to the statement that system (\ref{S2SS4:N1}) (with respect to $u$) has $q$ solutions, the set  $\{ w\in \F_q^k:  N_w = q\}$ is a subspace of $\F_q^k$.

Let $v_1 = (1,0,\cdots,0)$, $v_2 = (X,0,\cdots,0)$, and
$
v_i = (c_{i,1}f_i,\cdots, c_{i,d_g}f_i),\ 3\le i\le k
$.
Let ${\rm rank}(v_1,v_2,v_3\cdots,v_k)$ denote the dimension of the subspace generated by  $\{v_1,v_2,v_3\cdots,v_k\}$.  Then, we have,
$$
\left|\{ w\in \F_q^k:  N_w = q\}\right|= q^{k-{\rm rank}(v_1,v_2,v_3\cdots,v_k)}.
$$

It is clear that if one of the two conditions in the statement of the theorem is satisfied, then $v_1,v_2,v_3,\cdots, v_k$ are
$\F_q$-linearly independent, and hence $$\rank (v_1,v_2,v_3,\cdots,v_k) = k.$$ Therefore, the graph $S(k,q)$ is connected.
\end{proof}

We are ready to prove Theorem \ref{aa}.
\begin{proof}[\textbf{Proof of Theorem \ref{aa}}]
 Theorem \ref{aa} is an immediate corollary of Theorem \ref{S2SS4:largegap} and Theorem \ref{S2SS4:connectedness}.
\end{proof}

We conclude this section with an example of families of expanders. Their expansion properties follow from Theorem \ref{aa}.
\begin{example}
{\rm Fix $k\ge 3$. Choose $(b_n)_{n\ge 1}$ and $(c_n)_{n\ge 1}$ to be two increasing sequences of positive real numbers such that $b_n = o(n)$, and $c_n = o(\sqrt{n})$.

Let $(q_m)_{m\geq 1}$ be an increasing
 sequence of prime powers such that $b_{q_1} \ge k$.

Let $f_{3,m},\ldots, f_{k,m}$ be such that
$1,X,f_{3,m},\ldots,f_{k,m}$ are $\mathbb{F}_q$-linearly independent and $1\le d_f^{(m)} < b_q$.
Let $g_{3,m},\ldots,g_{k,m}$ be such that $g_{i,m}(-X) = -g_{i,m}(X)$ for each $i$, the coefficient of $X$ in $g_{i,m}$ is non-zero and $1\le d_g^{(m)} < \min (p_m,c_{q_m})$.
Then the graphs $S(k,q_m) = S(k,q_m; f_{3,m},g_{3,m},\cdots,f_{k,m}, g_{k,m})$, $m\ge 1$, form a family of expanders.}
\end{example}

\section{Spectra of the graphs $S(k,q)$ for $g_i(X)=X^{3}$}\label{ODDEXP}
In this section, we provide some specializations of the graphs $S(k,q)$ for $g_i(X) = X^3$, $3\le i\le k$,
and bound or compute their eigenvalues. Our goal is to prove Theorems \ref{aaaa} and \ref{aaaaa}.

\begin{lemma}\label{cubelemma}
Let $q$ be an odd prime power with $q\equiv 2\mod 3$ and $k\ge 3$. Suppose that
$g_i(X) =X^3$ for any $i$, $3\le i\le k$. For any $w\in \F_q^k$,
let $T_w$ be the number of $u\in\F_q$ such that $f_3(u)w_3+\cdots+f_k(u)w_k = 0$.
Then $\lambda_w$ is either $q(T_w-1)$ or at most $(q-T_w)M_q$, where $M_q = \max\limits_{a,b\in\F_q^*}\varepsilon_{ax^3+bx}\le 2\sqrt{q}$.
\end{lemma}
\begin{proof}
By (\ref{S2SS4:eignoclosed}), we have the following,
$$
\lambda_w = \sum\limits_{a\in\mathbb{F}_q^*,u\in\mathbb{F}_q} \zeta_p^{\ Tr\bigl(a(w_1+uw_2)+a^3\sum\limits_{i=3}^k f_i(u)w_i\bigr)},
$$
for any $w = (w_1,\cdots,w_k)$. Let
$F(X) = f_3(X)w_3+\cdots+f_k(X)w_k$.
\begin{enumerate}
\item[Case 1:] For
 $w$ of the form $(0,0,w_3,\cdots,w_k)$, we have:

\begin{align*}
\lambda_w &= \sum\limits_{\substack{u\in\F_q\\F(u) = 0}}\sum\limits_{a\in\F_q^*} \zeta_p^{\ Tr\bigl( a^3F(u)\bigr)} + \sum\limits_{\substack{u\in\F_q\\F(u) \ne 0}}\sum\limits_{a\in\F_q^*} \zeta_p^{\ Tr\bigl( a^3F(u)\bigr)}\\
                &= (q-1)T_w +\sum\limits_{\substack{u\in\F_q\\F(u) \ne 0}}\sum\limits_{a\in\F_q^*} \zeta_p^{\ Tr\bigl( a^3F(u)\bigr)}.
\end{align*}
Since $q\equiv 2 $ mod 3,  it follows that $\gcd(q-1,3)= 1$,  and $a\mapsto a^3$ defines  a bijection  of $\F_q$.  Therefore  the above term $\sum\limits_{a\in\F_q^*} \zeta_p^{\ Tr\bigl( a^3F(u)\bigr)}$ equals  $-1$.
 Hence,
$$\lambda_w = (q-1)T_w - (q-T_w) = q(T_w-1).$$
\item[Case 2:] For those $w$ of the form $(w_1,0,w_3,\cdots,w_k)$ with $w_1\ne 0$, we have,
\begin{align*}
\lambda_w &= \sum\limits_{\substack{u\in\F_q\\F(u) = 0}}\sum\limits_{a\in\F_q^*} \zeta_p^{\ Tr\bigl( aw_1+a^3F(u)\bigr)} + \sum\limits_{\substack{u\in\F_q\\F(u) \ne 0}}\sum\limits_{a\in\F_q^*} \zeta_p^{\ Tr\bigl( aw_1+a^3F(u)\bigr)}\\
                &= -T_w +\sum\limits_{\substack{u\in\F_q\\F(u) \ne 0}}(\varepsilon_{w_1a+F(u) a^3}-1) \\
                &= -q + \sum\limits_{\substack{u\in\F_q\\F(u) \ne 0}}\varepsilon_{w_1a+F(u) a^3}\\
                &\le -q +\sum\limits_{\substack{u\in\F_q\\F(u)\ne 0}} M_q \hskip 20mm\text{(by Lemma \ref{oddexp}, $\varepsilon_{w_1a+F(u)a^3}$ is real)}\\
                &\le -q + (q-T_w)M_q
                < (q-T_w)M_q.
\end{align*}

\item[Case 3:] For those $w$ of the form $w = (w_1,w_2,w_3,\cdots,w_k)$ with $w_2\ne 0$, we have

\begin{align}
\lambda_w & =  \sum\limits_{\substack{u\in\F_q\\F(u) = 0\\w_1+uw_2 = 0}}\sum\limits_{a\in\F_q^*} \zeta_p^{\ Tr( aw_1+auw_2)} + \sum\limits_{\substack{u\in\F_q\\F(u) = 0\\w_1+uw_2 \ne 0}}\sum\limits_{a\in\F_q^*} \zeta_p^{\ Tr( aw_1+auw_2)}\label{12}\\
                &\quad\quad +\sum\limits_{\substack{u\in\F_q\\F(u) \ne 0}}\sum\limits_{a\in\F_q^*} \zeta_p^{\ Tr\bigl( a(w_1+uw_2)+a^3F(u)\bigr)}.\label{3}
\end{align}

If $F(-w_1/w_2) = 0$, then the number of $u\in\F_q$ such that $F(u) = 0$ and $w_1+uw_2 = 0$ is 1, and hence,
\begin{align*}
\lambda_w &= (q-1)-(T_w-1)+\sum\limits_{\substack{u\in\F_q\\F(u) \ne 0}}\sum\limits_{a\in\F_q^*} \zeta_p^{\ Tr\bigl( a(w_1+uw_2)+a^3F(u)\bigr)}\\
                 &=q-T_w+\sum\limits_{\substack{u\in\F_q\\F(u) \ne 0\\w_1+uw_2\ne 0}}(\varepsilon_{(w_1+uw_2)a+F(u) a^3}-1)\\
                 &=\sum\limits_{\substack{u\in\F_q\\F(u) \ne 0\\w_1+uw_2\ne 0}}\varepsilon_{(w_1+uw_2)a+F(u) a^3}\\
                 &\le (q-T_w)M_q.
\end{align*}
Now assume that $F(-w_1/w_2)\ne 0$. Then, $w_1+uw_2 \ne 0$ if $F(u) = 0$.
 Then the first double sum in (\ref{12}) has no terms, the second double sum in (\ref{12}) is equal to $T_w(-1)$, and splitting the  double sum in  (\ref{3}) into two double sums, we obtain:

\begin{align*}
\lambda_w         & = -T_w+\sum\limits_{\substack{u\in\F_q\\F(u) \ne 0\\w_1+uw_2 = 0}}\sum\limits_{a\in\F_q^*} \zeta_p^{\ Tr\bigl( a(w_1+uw_2)+a^3F(u)\bigr)} + \sum\limits_{\substack{u\in\F_q\\F(u) \ne 0\\w_1+uw_2 \ne 0}}\sum\limits_{a\in\F_q^*} \zeta_p^{\ Tr\bigl( a(w_1+uw_2)+a^3F(u)\bigr)}\\
                 & = -T_w+ \sum\limits_{a\in\F_q^*} \zeta_p^{\ Tr\bigl(a^3F(u)\bigr)} + \sum\limits_{\substack{u\in\F_q\\F(u) \ne 0\\w_1+uw_2 \ne 0}}\sum\limits_{a\in\F_q^*} \zeta_p^{\ Tr\bigl( a(w_1+uw_2)+a^3F(u)\bigr)}\\
                 & = -T_w - 1 + \sum\limits_{\substack{u\in\F_q\\F(u) \ne 0\\w_1+uw_2 \ne 0}} (\varepsilon_{(w_1+uw_2)a+F(u) a^3}-1)\\
                 & = -q + \sum\limits_{\substack{u\in\F_q\\F(u) \ne 0\\w_1+uw_2 \ne 0}} \varepsilon_{(w_1+uw_2)a+F(u) a^3}\\
                 & \le -q + (q-T_w-1)M_q <    (q-T_w)M_q.
\end{align*}
\end{enumerate}
As $q \equiv 2\mod 3$,  we have $\gcd(3,q) = 1$. By Theorem \ref{WeilBound}, $M_q\le 2\sqrt{q}$, and the lemma is proven.
\end{proof}

Now we prove Theorem \ref{aaaa}, where $f_i(X) =X^{i-1}$, for any $3\le i\le k$. In this case, we are able to determine their second largest eigenvalues.
\begin{proof}[\textbf{Proof of Theorem \ref{aaaa}}]
Since $3\le k\le q+1$, it follows that $X^2,X^3,\cdots,X^{k-1}$ are $\F_q$-linearly independent, and hence $S(k,q)$ is connected by
 Theorem \ref{S2SS4:connectedness}. For any $w = (w_1,\cdots,w_k) \in \F_q^k$, let $F(X) = X^2w_3+X^3w_4+\cdots+X^{k-1}w_k = X^2(w_3+Xw_4+\cdots+X^{k-3}w_k)$,
 which implies that $T_w$ (defined in the statement of Lemma~\ref{cubelemma}) is either $q$ or between $1$ and $k-2$.  By Lemma~\ref{cubelemma}, we
 have that if $\lambda_w$ is not $q(T_w-1)$, then it is at most $(q-T_w)M_q\le (q-1)M_q$. Therefore, we obtain:
 $$\lambda_2\big(S(k,q)\big) \le \max\{ q(k-3), (q-1)M_q\}.$$
Moreover,
if $k\ge 4$, then the above inequality becomes equality. Indeed, for any $w\in\F_q^k$ of the form $w =(0,w_2,0,w_4,0,\cdots,0)$, where $w_2,w_4\ne 0$, the following holds:
\begin{align*}
\lambda_w &= \sum\limits_{a\in\F_q^*,u\in\F_q} \zeta_p ^{\ Tr(auw_2 + a^3u^3w_4)}\\
                 & = \sum\limits_{x\in\F_q} \sum\limits_{\substack{a\in\F_q^*,u\in\F_q\\au = x}} \zeta_p^{\ Tr(w_2x + w_4x^3)}\\
                 & = \sum\limits_{x\in\F_q} (q-1)\zeta_p^{\ Tr(w_2x + w_4x^3)}\\
                 & = (q-1) \varepsilon_{w_2x+w_4x^3}.
\end{align*}
This implies that $$\max\limits_{\substack{w = (0,w_2,0,w_4,0,\cdots,0)\\w_2,w_4\ne 0}} \{\lambda_w\} = (q-1)M_q.$$ Therefore, we have
$\lambda_2 (S(k,q)) = \max\{q(k-3),(q-1)M_q\}$.
As $q\equiv 2 \pmod{3}$,  by Theorem \ref{WeilBound}, $M_q\le 2\sqrt{q}$.
\end{proof}

\begin{proof}[\textbf{Proof of Theorem \ref{aaaaa}}]
Since $3\le k\le e+2$,
 it follows that $X^p, \cdots, X^{p^{k-2}}$ are $\F_q$-linearly independent, and hence $S(k,q)$ is connected by Theorem \ref{S2SS4:connectedness}. For any $w = (w_1,\cdots,w_k) \in \F_q^k$, let $F(X) = X^pw_3+\cdots + X^{p^{k-2}}w_k = (X^p)w_3+\cdots + (X^p)^{p^{k-3}}w_k = Yw_3 + \cdots + Y^{p^{k-3}}w_k$ where $Y =X^p$. Since  $a\mapsto a^p$ defines  a bijection  on $\F_q$, it implies that
$T_w$ (defined here as the number of roots of $F(X)$ in $\F_q$),  is either $q$ or at most $p^{k-3}$. The statement of the theorem then follows from Lemma~\ref{cubelemma}.
\end{proof}

\section{Concluding remarks}\label{CONREM}
In this section, we make some remarks on several specializations of $S(k,q)$ considered in Section \ref{ODDEXP}.
\begin{rk}\label{NonBip}
{\rm As we mentioned in Section \ref{Intro}, for every $q$-regular bipartite graph $\Gamma$, every eigenvalue of $\Gamma^{(2)}$ should be at least $-q$. For graphs $S(3,q; x^2,x^3)$ for prime $q$ between 5 and 19, and for graphs $S(4,q; x^2,x^3,x^3,x^3)$ for prime $q$ between 5 and 13, our computations show that their smallest eigenvalues are strictly less than $-q$. This implies that these graphs are not distance two graphs of any $q$-regular bipartite graphs.
}
\end{rk}
\begin{rk}
{\rm
In Section \ref{ODDEXP}, we discussed the graphs $S(k,q)$ with $g_i(X) = X^3$.
Now assume that $n\ge 1$, and $g_i(X) = X^{2n+1}$ for all $i$, $3\le i\le k$.
For these graphs, Lemma \ref{cubelemma} can be generalized as follows:

\textit{Let $q$ be an odd prime power with $q\not\equiv 1\mod (2n+1)$ and $(2n+1,q) = 1$. For any $w\in\F_q^k$, let $N_w$ be the number of $u\in\F_q$ such that $w_3f_3(u)+\cdots+w_kf_k(u)=0$. Then $\lambda_w$ is either $q(N_w-1)$ or at most $2n(q-N_w)\sqrt{q}$.}

In the case when $3\le k\le q+1$, $f_i(X) =X^{i-1}$ and $g_i(X) = X^{2n+1}$ for all $i$, $3\le i\le k$, the conclusion of Theorem \ref{aaaa} can be stated in a slightly weaker form:
$$
\lambda_2\bigl(S(k,q)\bigr) \le \max\{q(k-3),2n(q-1)\sqrt{q}\}.
$$
Actually, for fixed $q$, if $k$ is sufficiently large, $\lambda_2\bigl(S(k,q)\bigr) = q(k-3)$ for all $n\ge 1$. }

\end{rk}

\begin{rk}{\rm
The quantity $M_q =\max\limits_{a,b\in\F_q^*} \varepsilon_{ax^3+bx}$ in Theorem \ref{aaaa} and Theorem \ref{aaaaa} is at most $2\sqrt{q}$ by Weil's bound. From the computational results, $M_q \ge 2\sqrt{q}-2$ for $q\le 1331$. Interestingly, when $q = 5^3$ or $5^5$, the Weil's bound is tight.}
\end{rk}

\begin{rk}{\rm
Let $k\geq 3$ be an integer and let $f_i,g_i\in \F_q[X]$, $3\leq i\leq k+1$, be $2k-2$ polynomials of degree at most $q-1$ such that $g_i(-X)=-g_i(X)$ for each $i$, $3\leq i\leq k+1$. If $S(k+1,q)=S(k+1,q;f_3,g_3,\dots,f_k,g_k,f_{k+1},g_{k+1})$ and $S(k,q)=S(k,q;f_3,g_3,\dots,f_k,g_k)$, then it is not hard to show that $S(k+1,q)$ is a $q$-cover of $S(k,q)$ (see, e.g.,  \cite[Section 6]{HLW}). This implies that the spectrum of $S(k+1,q)$ is a submultiset  the spectrum of $S(k,q)$ and, in particular, $\lambda_2(S(k+1,q))\geq \lambda_2(S(k,q))$.

Interestingly, in the case when $f_i(X) = X^{i-1}$ and $g_i(X) = X^3$ for each $i\ge 3$, we actually have equality in the inequality above for $(q,k)$ whenever $\displaystyle{k < \frac{q-1}{q}M_q+2}$ (immediate from Theorem \ref{aaaa}).
}
\end{rk}

\section{Acknowledgment}
The work of the first author was supported by the NSF grant DMS-1600768, and of the second author by the Simons Foundation grant \#426092.

\end{document}